\documentclass{amsart}
\usepackage{amsmath,amssymb,amsthm, tikz, enumerate}
\usepackage{xcolor}

\newtheorem{theorem}{Theorem}
\newtheorem{lemma}[theorem]{Lemma}

\newtheorem{conjecture}[theorem]{Conjecture}

\newtheorem{question}[theorem]{Question}
\theoremstyle{definition}

\newcommand{\co}{\colon\thinspace}

\newcommand{\abs}[1]{\left\lvert #1 \right\rvert}
\DeclareMathOperator{\Length}{Length}

\DeclareMathOperator{\Vol}{Vol}

\DeclareMathOperator{\sys}{sys}

\begin{document}

\title[Growth and simplicial volume]{Growth in the universal cover under\\ large simplicial volume}
\author{Hannah Alpert}
\address{Auburn University, 221 Parker Hall, Auburn, AL 36849}
\email{hcalpert@auburn.edu}

\begin{abstract}
Consider a closed manifold $M$ with two Riemannian metrics: one hyperbolic metric, and one other metric $g$.  What hypotheses on $g$ guarantee that for a given radius $r$, there are balls of radius $r$ in the universal cover of $(M, g)$ with greather-than-hyperbolic volumes?  We show that this conclusion holds for all $r \geq 1$ if $(\Vol (M, g))^2$ is less than a small constant times the hyperbolic volume of $M$.  This strengthens a theorem of Sabourau and is partial progress toward a conjecture of Guth.
\end{abstract}

\maketitle

\section{Introduction}

Let $M$ be a closed Riemannian manifold.  For each $r > 0$ we let $V_r(M)$ denote the supremal volume of a ball of radius $r$ in $M$.  We let $\Vert M \Vert_{\Delta}$ denote the Gromov simplicial volume of $M$; the definition appears in Section~\ref{sec:simp-vol}.  The relevant properties are that $\Vert M \Vert_{\Delta}$ depends only on the fundamental group of $M$, and if $M$ is hyperbolic then $\Vert M \Vert_{\Delta}$ is equal to the hyperbolic volume of $M$ times a constant depending only on the dimension of $n$.  The purpose of this paper is to prove the following theorem.

\begin{theorem}[Main theorem]\label{thm:main-cor}
For every $n \geq 2$, there exists $\delta_n > 0$ such that if $M$ is a closed, oriented, connected $n$-dimensional manifold admitting a hyperbolic metric, and $g$ is another Riemannian metric on $M$ with $\frac{(\Vol (M, g))^2}{\Vert M \Vert_{\Delta}} < \delta_n$, then for all $r \geq 1$ we have
\[V_r(\widetilde{M}) \geq V_r(\mathbb{H}^n),\]
where $\widetilde{M}$ denotes the universal cover of $(M, g)$, and $\mathbb{H}^n$ denotes the hyperbolic $n$-space.
\end{theorem}

Theorem~\ref{thm:main-cor} strengthens the following theorem of Sabourau and also uses a large part of its proof.

\begin{theorem}[\cite{sabourau22}]\label{thm:sabourau-cor}
For every $n \geq 2$, there exists $\delta_n > 0$ such that if $M$ is a closed, oriented, connected $n$-dimensional manifold admitting a hyperbolic metric, and $g$ is another Riemannian metric on $M$ with $\Vol (M, g) < \delta_n$, then for all $r \geq 1$ we have
\[V_r(\widetilde{M}) \geq V_r(\mathbb{H}^n).\]
\end{theorem}

Sabourau's proof uses Papasoglu's method of area-minimizing separating sets, introduced in~\cite{papasoglu20}.  The part of the proof of Theorem~\ref{thm:main-cor} that is different from the proof of Theorem~\ref{thm:sabourau-cor} is contained mainly in~\cite[Theorem 5]{alpert22}, which combines Papasoglu's method with simplicial volume to prove the following theorem, originally due to Guth.

\begin{theorem}[Theorem 2 of~\cite{guth11}]\label{thm:guth}
For every $n \geq 2$, there exists $\delta_n > 0$ such that if $M$ is a closed, oriented, connected $n$-dimensional manifold admitting a hyperbolic metric, and $g$ is another Riemannian metric on $M$ with $\frac{\Vol (M, g)}{\Vert M \Vert_{\Delta}} < \delta_n$, then we have
\[V_1(\widetilde{M}) \geq V_1(\mathbb{H}^n).\]
\end{theorem}

The most ambitious conjecture along these lines is~\cite[Conjecture 2]{guth11}, which says that if $\Vol(M, g) < \Vol (M, \mathrm{hyp})$, then $V_r(\widetilde{M}) > V_r(\mathbb{H}^n)$ for all $r$.  This statement strengthens the theorem of Besson, Courtois, and Gallot which proves the conclusion for all sufficiently large $r$~\cite{besson95}.  A more feasible-sounding next step, suggested by Guth, would be to combine the conclusion of Theorems~\ref{thm:main-cor} and~\ref{thm:sabourau-cor} with the hypothesis of Theorem~\ref{thm:guth}.

\begin{conjecture}[\cite{guth11}]\label{conj:main}
For every $n \geq 2$, there exists $\delta_n > 0$ such that if $M$ is a closed, oriented, connected $n$-dimensional manifold admitting a hyperbolic metric, and $g$ is another Riemannian metric on $M$ with $\frac{\Vol (M, g)}{\Vert M \Vert_{\Delta}} < \delta_n$, then for all $r \geq 1$ we have
\[V_r(\widetilde{M}) \geq V_r(\mathbb{H}^n).\]
\end{conjecture}

Karam has proved the $n=2$ case of Conjecture~\ref{conj:main} in~\cite{karam15}.  In Section~\ref{sec:open} we speculate on what would be needed to prove Conjecture~\ref{conj:main} if combining Karam's methods with those of the present paper.

Section~\ref{sec:sabourau} summarizes the proof of Theorem~\ref{thm:main-cor} and explains which parts are taken from~\cite{sabourau22}.  Section~\ref{sec:simp-vol} presents the property of simplicial volume needed for the proof.  Section~\ref{sec:papasoglu} presents the use of Papasoglu's method from~\cite{alpert22}.  Section~\ref{sec:main} completes the proof of Theorem~\ref{thm:main-cor}, and Section~\ref{sec:open} discusses open questions.

\section{Relationship with Sabourau's proof}\label{sec:sabourau}

In this section we give an overview of the ingredients needed for the proof of Theorem~\ref{thm:main-cor}.  This includes stating the parts of~\cite{sabourau22} that we use unaltered, and commenting on the changes in the remainder of the proof.

To get the comparison $V_r(\widetilde{M}) \geq V_r(\mathbb{H}^n)$ for all $r \geq 1$, we need to show that $V_r(\widetilde{M})$ has at least the same exponential growth rate as $V_r(\mathbb{H}^n)$.  Sabourau's argument relies on the following theorem.

\begin{theorem}[\cite{dey19}]\label{thm:tits}
For each $n$, there exists $\kappa$ such that the fundamental group $G$ of every closed, connected, hyperbolic $n$-dimensional manifold satisfies the $\kappa$-Tits alternative: for every symmetric subset $S$ of $G$ containing the identity, either $S$ generates a group of subexponential growth, or there exist two elements in $S^{\kappa}$ generating a nonabelian free subgroup.
\end{theorem}

On any Riemannian manifold $M$, the \textbf{\textit{Margulis function}} $\mu \co M \rightarrow \mathbb{R} \cup \{\infty\}$ expresses the threshold between subexponential growth and exponential growth.  At each point $x \in M$, the function value $\mu(x)$ is defined to be the supremal value $\mu$ such that the subgroup of $\pi_1(M, x)$ generated by the loops based at $x$ with length at most $\mu$ is a group of subexponential growth.  The function $\mu(x)$ is greater than or equal to the \textbf{\textit{systole}} of $M$, denoted by $\sys M$, which is the infimal length of a homotopically nontrivial loop in $M$.  For $\rho < \frac{1}{2}\mu(x)$, the image in $\pi_1(M)$ of the ball $B(x, \rho)$ has subexponential growth and thus is amenable; for $\rho > \frac{1}{2}\mu(x)$, Theorem~\ref{thm:tits} exhibits two loops at $x$ of length at most $2\rho\kappa$ that generate a nonabelian free subgroup of $\pi_1(M)$, and so the volumes of balls in $\widetilde{M}$ have at least a certain exponential growth rate.  Sabourau uses this idea to prove the following lemma.

\begin{lemma}[\cite{sabourau22}]\label{lem:ping-pong}
For every $n \geq 2$, there exists a small number $\alpha_n > 0$ such that the following holds.  Let $M$ be a closed, oriented, connected $n$-dimensional manifold admitting a hyperbolic metric, and let $g$ be another Riemannian metric on $M$.  Suppose that there exists $x_0 \in M$ such that $\frac{1}{2}\mu(x_0) < \alpha_n$ and $\Vol B\left(x_0, \frac{1}{2}\mu(x_0)\right) \geq \displaystyle \frac{\left(\frac{1}{2}\mu(x_0)\right)^n}{n!}$. Then for all $r \geq 1$ we have $V_r(\widetilde{M}) \geq V_r(\mathbb{H}^n)$.
\end{lemma}

Because Sabourau does not state it in this way, we include the following notes on how to deduce this statement from the proofs of~\cite[Theorem 3.7, Lemma 3.8, Corollaries 3.9 and 3.10]{sabourau22}.

\begin{proof}[Proof summary]
As in~\cite[Corollary 3.10]{sabourau22}, Theorem~\ref{thm:tits} implies that $M$ satisfies the $\kappa$-Tits alternative for some $\kappa$ depending only on $n$.  Let $\mu_0 = \mu(x_0)$, and let $\widetilde{x}_0$ denote any lift of $x_0$ to $\widetilde{M}$.  In our notation, \cite[Lemma 3.8]{sabourau22} says that because of our hypothesis $\Vol B\left(x_0, \frac{1}{2}\mu_0\right) \geq \frac{\left(\frac{1}{2}\mu_0\right)^n}{n!}$ we have
\[\Vol B_{\widetilde{M}}(\widetilde{x}_0, r) \geq \frac{\left(\frac{1}{2}\mu_0\right)^n}{n!} \cdot \left(3^{\left\lfloor\frac{r-\frac{\mu_0}{2}}{2\kappa\mu_0}\right\rfloor}-1\right)\]
for all $r \geq \frac{1}{2}\mu_0$.  The proof is unchanged, depending only on the definition of $\kappa$-Tits alternative. For sufficiently small $\alpha_n$ we have $\mu_0 < 5\kappa\mu_0 < 1$, so the proof of Case 1 of~\cite[Theorem 3.7]{sabourau22} gives a constant $C_{n, \kappa}$ such that
\[\Vol B_{\widetilde{M}}(\widetilde{x}_0, r) \geq \Vol B_{\mathbb{H}^n}(C_{n, \kappa} r)\]
for all $r \geq 1$.  Then as in~\cite[Corollary 3.9]{sabourau22} we can use $C_{n, \kappa}$ to find a large constant $\lambda_{n, \kappa}$ such that applying the above inequality to $\lambda_{n, \kappa}M$ implies the desired conclusion
\[\Vol B_{\widetilde{M}}(\widetilde{x}_0, r) \geq \Vol B_{\mathbb{H}^n}(r)\]
for all $r \geq 1$, using the value of $\lambda_{n, \kappa}$ to select $\alpha_n$ sufficiently small.
\end{proof}

The remainder of Theorem~\ref{thm:sabourau-cor} is proved by showing in~\cite[Theorem 2.7 and Proposition 3.3]{sabourau22} that the assumption $\Vol (M, g) < \delta_n$ implies the existence of a point $x_0$ with the two specified properties.  We prove Theorem~\ref{thm:main-cor} by showing that the assumption $\frac{(\Vol(M, g))^2}{\Vert M \Vert_{\Delta}} < \delta_n$ is sufficient to guarantee such a point $x_0$.

We informally summarize the proof of Theorem~\ref{thm:main-cor} as follows.  Instead of producing only one point $x_0$, for each $\varepsilon > 0$ we produce a finite set of points $Z_0(\varepsilon) \subseteq M$ with the following properties:
\begin{enumerate}
\item $\Vol B(x, \frac{1}{8}\mu(x)) \gtrsim \mu(x)^n - \varepsilon$ for each $x \in Z_0(\varepsilon)$, by Lemma~\ref{lem:volball}.  This plays the role of the second property of $x_0$ in Lemma~\ref{lem:ping-pong}, and the proof is the same as Sabourau uses.
\item $\Vol M \geq \Vol B\left(x, \frac{1}{2}\mu(x)\right) \gtrsim \mu(x)^n \cdot \#\left(Z_0 \cap B\left(x, \frac{1}{4}\mu(x)\right)\right) - \varepsilon$ for each $x \in Z_0(\varepsilon)$, by Lemma~\ref{lem:volball}.  This property compensates for the fact that the balls around various points of $Z_0(\varepsilon)$ may overlap.
\item The number of points in $Z_0(\varepsilon)$ is at least $\frac{1}{2^n} \cdot \Vert M \Vert_{\Delta}$, by Theorem~\ref{thm:amenable-reduction} and Lemma~\ref{lem:triangulate}.  In contrast, Sabourau's proof only guarantees that the number of points in $Z_0(\varepsilon)$ is nonzero.  This statement that a large value of $\Vert M \Vert_{\Delta}$ implies a large number of points in $Z_0(\varepsilon)$, and not just one point, is the main new idea in our proof.
\end{enumerate}
The proof of Theorem~\ref{thm:main-cor}, given in Section~\ref{sec:main}, consists of summing the volume estimates over a suitable subset of $Z_0(\varepsilon)$ and taking the limit as $\varepsilon$ approaches zero.  To get the idea, the reader may safely ignore all instances of $\varepsilon$ on a first read.

\section{Amenable Reduction Lemma}\label{sec:simp-vol}

The simplicial norm was introduced by Gromov in~\cite{gromov82}.  Let $z = \sum_i a_i \sigma_i$ be a singular $d$-cycle on a space $P$, with real coefficients $a_i \in \mathbb{R}$ and simplices $\sigma_i \co \Delta^d \rightarrow P$.  The $\ell^1$ norm of $z$, denoted by $\abs{z}_1$, is $\sum_i \abs{a_i}$, and the \textit{\textbf{simplicial norm}} of a given homology class is the infimum of $\abs{z}_1$ over all cycles $z$ representing the homology class.  The simplicial norm of the class $[z]$ is denoted by $\Vert [z] \Vert_{\Delta}$.  The \textit{\textbf{simplicial volume}} of a closed, oriented, connected manifold $M$, denoted by $\Vert M \Vert_{\Delta}$, is the simplicial norm of the fundamental homology class of $M$ with $\mathbb{R}$-coefficients.  

The property of simplicial norm that we use in this paper is Theorem~\ref{thm:amenable-reduction}, Gromov's Amenable Reduction Lemma.  To understand the relationship between simplicial norm and amenable groups, it may be helpful to start with Gromov's Vanishing Theorem, found in~\cite[Section 3.1]{gromov82} or~\cite[Corollary 6.3]{ivanov87}.  The Vanishing Theorem can be viewed as a special case of the Amenable Reduction Lemma.

To state Theorem~\ref{thm:amenable-reduction}, we need to define what it means to color the vertices of a singular cycle.  Given an arbitrary singular $d$-cycle $z = \sum_{i} a_i\sigma_i$ in a space $P$, we form a $\Delta$-complex in the following way.  We take one Euclidean $d$-simplex for each $\sigma_i$, and identify the $(d-1)$-dimensional faces of these simplices whenever the corresponding singular $(d-1)$-simplices are equal.  Because $\partial z = 0$, every $(d-1)$-dimensional face is identified with at least one other face.  This complex has a simplicial $d$-cycle in which the coefficient of each simplex equals the coefficient in $z$ of the corresponding singular simplex.  We can think of $z$ as the image of this simplicial cycle under the map from the $\Delta$-complex into $P$, and we can talk about the vertices and edges of $z$, meaning the vertices and edges of the $\Delta$-complex, equipped with their maps into $P$.

If $P$ is path-connected, we define a \textbf{\textit{amenable vertex coloring}} of $z$ to be a way to assign colors to the vertices of $z$ such that for each color, if we take the union of all edges of $z$ for which both vertices are that color, then the subgroup of $\pi_1(P)$ generated by this $1$-complex is amenable.  In an amenable vertex coloring, we also require that $z$ does not contain any edges from a vertex to itself.  (In this case, coloring every vertex a different color gives, trivially, an amenable vertex coloring.)  We define a \textbf{\textit{rainbow simplex}} of such a coloring to be any simplex in $z$ for which all $d+1$ vertices are different colors.  The following theorem is from~\cite[Section 3.2]{gromov09} (or see~\cite{alpert16}).

\begin{theorem}[Amenable Reduction Lemma]\label{thm:amenable-reduction}
Let $z = \sum_i a_i\sigma_i$ be a singular cycle on a path-connected cell complex $P$, with an amenable vertex coloring.  Then the simplicial norm of the homology class of $z$ satisfies
\[\Vert [z] \Vert_{\Delta} \leq \sum_{\mathrm{rainbow}\ \sigma_i} \abs{a_i}.\]
\end{theorem}

\section{Separating filtrations}\label{sec:papasoglu}

This section contains lemmas from~\cite{alpert22} that are based on Papasoglu's method of area-minimizing separating sets, which was introduced in~\cite{papasoglu20}.  To learn the method, another useful exposition is by Nabutovsky in~\cite[Theorem 1.1]{nabutovsky19}, which has as a corollary Gromov's Systolic Inequality, first proven in~\cite[Theorem 0.1.A]{gromov83}

Let $M$ be a closed $n$-dimensional Riemannian manifold, and let $\rho \co M \rightarrow \mathbb{R}$ be a positive, bounded function.  A  \textbf{\textit{$\rho$-separating filtration}} of $M$ consists of sets
\[M = Z_n \supseteq Z_{n-1} \supseteq \cdots \supseteq Z_1 \supseteq Z_0 \supseteq Z_{-1} = \emptyset,\]
such that for each $i = 0, 1, \ldots, n$, the set $Z_{i-1}$ is a \textbf{\textit{$\rho$-separating subset}} of $Z_i$; that is, for each connected component of $Z_i \setminus Z_{i-1}$, there is some $x \in M$ such that the ball $B(x, \rho(x))$ contains this connected component.  (Henceforth we do not include $Z_{-1}$ in the notation.)  We also require smoothness and transversality conditions, given in~\cite{alpert22}, which guarantee in particular that $M$ admits a triangulation in which each $Z_i$ is an $i$-dimensional subcomplex.  These conditions ensure the following lemma.

\begin{lemma}[Lemma 4 of~\cite{alpert22}]\label{lem:triangulate}
Let $M$ be a closed, connected $n$-dimensional Riemannian manifold, and let $\rho \co M \rightarrow (0, \infty)$ be a bounded function with the property that for every ball $B(x, \rho(x))$ in $M$, its image in $\pi_1(M)$ is amenable.  Then for every $\rho$-separating filtration 
\[M = Z_n \supseteq Z_{n-1} \supseteq \cdots \supseteq Z_1 \supseteq Z_0\]
there is a triangulation of $M$ with an amenable vertex coloring, such that the number of rainbow simplices is $2^n \cdot \#Z_0$.
\end{lemma}

We omit the proof, which can be obtained from the proof in~\cite{alpert22} by replacing the constant radius $R$ by the varying radius $\rho$ and by requiring the contributions to $\pi_1(M)$ to be amenable instead of trivial.  The process for constructing the triangulation from the filtration is the same in both proofs.

The following lemma shows that it is possible to choose a $\rho$-separating filtration such that the balls around the points of $Z_0$ satisfy lower bounds on their volumes.

\begin{lemma}[Lemma 7 of~\cite{alpert22}]\label{lem:volball}
Let $M$ be a closed $n$-dimensional Riemannian manifold.  For every positive, bounded function $\rho \co M \rightarrow \mathbb{R}$ and every $\varepsilon > 0$, there exists a $\rho$-separating filtration
\[M = Z_n \supseteq Z_{n-1} \supseteq \cdots \supseteq Z_1 \supseteq Z_0,\]
such that for all $x \in M$ and all $r_1, r_2$ with $0 < r_1 < r_2 < \rho(x)$ we have
\[\#(Z_0 \cap B(x, r_1)) \cdot \frac{(r_2 - r_1)^n}{n!} \leq \Vol B(x, r_2) + \varepsilon.\]
\end{lemma}

We omit the proof, which can be obtained from the proof in~\cite{alpert22} by replacing the constant radius $R$ by the varying radius $\rho$, both in this lemma and in the preceding Lemma~6 of~\cite{alpert22}.  When computing the choice of how close each $Z_i$ should be to the infimal possible area, we use $\sup \rho$ instead of $R$.

\section{Main proof}\label{sec:main}

The proof of Theorem~\ref{thm:main-cor} is obtained by combining Lemma~\ref{lem:ping-pong} with the following Theorem~\ref{thm:main}, which is similar both in statement and in proof to \cite[Theorem 5]{alpert22}, but where the role of $\frac{1}{2}\sys M$ is replaced by $r_0(M)$, defined as follows.  Given a closed, connected $n$-dimensional Riemannian manifold $M$, we define the ball growth threshold radius $r_0(M)$, or \textbf{\textit{growth radius}} for short, to be the infimum of $\frac{1}{2}\mu(x)$, taken over only those $x\in M$ that satisfy the property that for all $r \leq \frac{1}{2}\mu(x)$, we have $\Vol B(x, r) \geq \frac{r^n}{n!}$.  Here $\mu$ denotes the Margulis function as defined in Section~\ref{sec:sabourau}.  Sabourau's proof of Theorem~\ref{thm:sabourau-cor} shows that if $\Vert M \Vert_{\Delta} > 0$, then $r_0(M)$ is finite.

\begin{theorem}\label{thm:main}
Let $M$ be a closed, oriented, connected $n$-dimensional Riemannian manifold.  Then we have
\[\Vert M \Vert_{\Delta} \leq 16^n(n!)^2 \cdot \left(\frac{\Vol M}{r_0(M)^n}\right)^2.\]
\end{theorem}

Assuming Theorem~\ref{thm:main}, we can quickly finish the proof of Theorem~\ref{thm:main-cor}.

\begin{proof}[Proof of Theorem~\ref{thm:main-cor}]
Using the $\alpha_n$ guaranteed by Lemma~\ref{lem:ping-pong}, we choose $\delta_n$ to be
\[\delta_n = \frac{(\alpha_n)^{2n}}{16^n(n!)^2}.\]
Then if $\displaystyle\frac{(\Vol M)^2}{\Vert M \Vert_{\Delta}} < \delta_n$, using Theorem~\ref{thm:main} we have
\[\frac{r_0(M)^{2n}}{16^n(n!)^2} \leq \frac{(\Vol M)^2}{\Vert M \Vert_{\Delta}} < \delta_n = \frac{(\alpha_n)^{2n}}{16^n(n!)^2},\]
and thus $r_0(M) < \alpha_n$.  Applying Lemma~\ref{lem:ping-pong} completes the proof.
\end{proof}

\begin{proof}[Proof of Theorem~\ref{thm:main}]
We select a sequence $\varepsilon_k \rightarrow 0$ and apply Lemma~\ref{lem:volball} with $\varepsilon = \varepsilon_k$ and $\rho = \frac{1}{2}\mu$.  In the resulting $\rho$-separating filtration, we define $\rho_k$ by
\[\rho_k = \min_{x \in Z_0}\rho(x).\]
We select a subset $S$ of $Z_0$ in the following way.  We consider the points of $Z_0$ in order from greatest to least value of $\rho$, and for each point $x \in Z_0$, we include it in $S$ if the ball $B(x, \frac{1}{4}\rho(x))$ is disjoint from all balls $B(s, \frac{1}{4}\rho(s))$ as $s$ ranges over the points already selected to be in $S$.  Then the balls $\{B(s, \frac{1}{4}\rho(s)) : s \in S\}$ are disjoint and the balls $\{B(s, \frac{1}{2}\rho(s)) : s \in S\}$ cover $Z_0$.

Applying Lemma~\ref{lem:volball} with $r_1 = \frac{1}{2}\rho$ and $r_2 \rightarrow \rho$, for each $s \in S$ we have
\[\#\left(Z_0 \cap B\left(s, \frac{1}{2}\rho(s)\right)\right) \leq (\Vol B(s, \rho(s)) + \varepsilon_k)\cdot \frac{n!}{\left(\frac{1}{2}\rho(s)\right)^n} \leq (\Vol M + \varepsilon_k) \cdot \frac{n!}{(\frac{1}{2}\rho_k)^n}.\]
Applying Lemma~\ref{lem:volball} with $r_1 \rightarrow 0$ and $r_2 = \frac{1}{4}\rho$, for each $s \in S$ we have
\[1 \leq \frac{\Vol B\left(s, \frac{1}{4}\rho(s)\right)}{\left(\frac{\left(\frac{1}{4}\rho(s)\right)^n}{n!} - \varepsilon_k\right)} \leq \frac{\Vol B\left(s, \frac{1}{4}\rho(s)\right)}{\left(\frac{\left(\frac{1}{4}\rho_k\right)^n}{n!} - \varepsilon_k\right)}.\]
Combining these inequalities with Lemma~\ref{lem:triangulate} and Theorem~\ref{thm:amenable-reduction}, we have
\begin{align*}
\Vert M\Vert_{\Delta} &\leq 2^n \cdot \#Z_0 \leq\\
&\leq 2^n \cdot \sum_{s \in S} \#\left(Z_0 \cap B\left(s, \frac{1}{2}\rho(s)\right)\right) \leq\\
&\leq 2^n \cdot \sum_{s \in S} (\Vol M + \varepsilon_k) \cdot \frac{n!}{(\frac{1}{2}\rho_k)^n} \leq\\
&\leq 2^n \cdot (\Vol M + \varepsilon_k) \cdot \frac{n!}{(\frac{1}{2}\rho_k)^n} \cdot \sum_{s \in S} \frac{\Vol B\left(s, \frac{1}{4}\rho(s)\right)}{\left(\frac{\left(\frac{1}{4}\rho_k\right)^n}{n!} - \varepsilon_k\right)} \leq\\
& \leq 2^n \cdot (\Vol M + \varepsilon_k) \cdot \frac{n!}{(\frac{1}{2}\rho_k)^n} \cdot \frac{\Vol M}{\left(\frac{\left(\frac{1}{4}\rho_k\right)^n}{n!} - \varepsilon_k\right)}.
\end{align*}
Let $x_k$ be a point in $Z_0$ with $\rho(x_k) = \rho_k$.  Because $M$ is compact, by passing to a subsequence we may assume that the sequence $x_k$ converges to a point $x \in M$.  The function $\rho = \frac{1}{2}\mu$ is continuous; in fact it is $1$-Lipschitz.  Thus $\rho_k \rightarrow \rho(x)$ and taking the limit we have
\[\Vert M\Vert_{\Delta} \leq 16^n(n!)^2 \cdot \left(\frac{\Vol M}{(\rho(x))^n}\right)^2.\]
What remains is to show that $\rho(x) \geq r_0(M)$.  For each $x_k$, for all $r \leq \rho(x_k)$, by Lemma~\ref{lem:volball} we have
\[\frac{r^n}{n!} \leq \Vol B(x_k, r) + \varepsilon_k.\]
We need to show that for all $r \leq \rho(x)$, we have
\[\frac{r^n}{n!} \leq \Vol B(x, r).\]
If so, then $\rho(x) \geq r_0(M)$ by definition.

For all $r$, we have
\[B(x_k, r - d(x, x_k)) \subseteq B(x, r).\]
By the $1$-Lipschitz property of $\rho$, if $r \leq \rho(x)$, then $r - d(x, x_k) \leq \rho(x_k)$.  Thus we have
\[\frac{(r - d(x, x_k))^n}{n!} \leq \Vol B(x_k, r - d(x, x_k)) + \varepsilon_k \leq \Vol B(x, r) + \varepsilon_k,\]
and taking the limit gives
\[\frac{r^n}{n!} \leq \Vol B(x, r),\]
as desired.
\end{proof}

\section{Open questions}\label{sec:open}

In this open questions section, first we propose some possible strengthenings of Theorem~\ref{thm:main}, and then we propose some possible statements that could prove Conjecture~\ref{conj:main}.  

Theorem~\ref{thm:main} is stronger than the same statement with $r_0(M)$ replaced by $\frac{1}{2}\sys M$.  A different (asymptotic) strengthening of that statement about $\sys M$ is the following theorem of Gromov.

\begin{theorem}[Theorem 6.4.D' of \cite{gromov83}]\label{thm:gromov}
There exists $C_n > 0$ such that if $M$ is a closed, oriented, connected $n$-dimensional Riemannian manifold, then we have
\[\Vert M \Vert_{\Delta} \leq C_n \frac{\Vol M}{(\sys M)^n} \cdot \left(\log\left(C_n\frac{\Vol M}{(\sys M)^n}\right)\right)^n.\]
\end{theorem}

We can ask whether there is an analogous further strengthening of Theorem~\ref{thm:main}.

\begin{question}\label{ques:r0}
Does Theorem~\ref{thm:gromov} remain true if $\sys M$ is replaced by $r_0(M)$?
\end{question}

The following further strengthening of Theorem~\ref{thm:gromov} would imply Conjecture~\ref{conj:main}.

\begin{question}\label{ques:ball}
Is the following statement true?  There exists $C_n > 0$ such that if $M$ is a closed, oriented, connected $n$-dimensional Riemannian manifold, then for all $r \leq \frac{1}{2}\sys M$, we have
\[\Vert M \Vert_{\Delta} \leq C_n\frac{\Vol M}{r^n}\cdot \left(\log\left(C_n\frac{V_r(M)}{r^n}\right)\right)^n.\]
\end{question}

The proof of Theorem~\ref{thm:gromov} constructs a metric minimizing the value of $\frac{\Vol M}{(\sys M)^n}$ and then proves the statement above for this special metric.  Thus, it is reasonable to hope that the same statement may be true of an arbitrary metric.  If so, then we would have
\[C_n \frac{V_r(M)}{r^n} \geq \exp\left(r \left(\frac{\Vert M\Vert_{\Delta}}{C_n \Vol M}\right)^{\frac{1}{n}}\right),\]
meaning that if $\frac{\Vol M}{\Vert M \Vert_{\Delta}}$ is small, then $V_r(M)$ has a high exponential growth rate as long as $r \leq \frac{1}{2}\sys M$.  Taking a covering space in which $\sys M$ is arbitrarily large, this statement would imply Conjecture~\ref{conj:main}.

A second potential approach to Conjecture~\ref{conj:main} is to try to extend Karam's proof of the $n=2$ case to higher dimensions, as follows.

\begin{question}\label{ques:spiderweb}
Is the following statement true?  There exist $C_n, \varepsilon_n > 0$ such that if $M$ is a closed, oriented, connected $n$-dimensional Riemannian manifold, then $M$ has a finite-sheeted covering space $\widehat{M}$ such that the following holds.  Suppose that in the $(\frac{1}{2}\sys \widehat{M})$-separating filtration
\[\widehat{M} = Z_n \supseteq Z_{n-1} \supseteq \cdots \supseteq Z_1 \supseteq Z_0,\]
for every $i = 0, \ldots, n-1$, the set $Z_i$ is within $\varepsilon_n$ of the infimal $i$-dimensional area of $(\frac{1}{2}\sys \widehat{M})$-separating subsets of $Z_{i+1}$.  Then
\[\Length Z_1 \leq C_n\cdot\left(\Vol \widehat{M} + \Vert \widehat{M}\Vert_{\Delta}\right).\]
\end{question}

We quickly sketch a proof, using the method of~\cite{karam15}, that an affirmative answer to this question would imply Conjecture~\ref{conj:main}.  By a rescaling argument similar to~\cite[Corollary 3.9]{sabourau22}, it suffices to show that there are constants $a_n, R_n, \delta'_n > 0$ such that if $\frac{\Vol M}{\Vert M \Vert_{\Delta}} < \delta'_n$, then $V_r(\widetilde{M}) \geq a_n \cdot 2^r$ for all $r \geq R_n$.  To do this, in a cover $\widehat{M}$ with $\frac{1}{2}\sys\widehat{M} > r$ we find a subset $Z'_1$ of $Z_1$ that, when viewed as a graph, has minimum vertex degree at least $3$, has no null-homotopic cycles, and has at least $\frac{\Vert \widehat{M}\Vert_{\Delta}}{2^n}$ edges.  Thus, if $\frac{\Vol M}{\Vert M \Vert_{\Delta}} \leq \left(\frac{1}{12 \cdot 2^n C_n}\right)^n$, then the average edge length in $Z'_1$ is at most $\frac{1}{6}$.  Then \cite[Theorem C]{karam15} implies that there is a point $x_0$ in $Z'_1$ such that for all $r$ with $1 \leq r \leq \frac{1}{2}\sys \widehat{M}$, we have
\[\Length Z'_1 \cap B(x_0, r) \geq \frac{3}{4}\cdot 2^r.\]
Using the near-minimizing property of each $Z_i$ as in~\cite[Lemma 6]{alpert22}, this implies by induction that the $i$-dimensional area of $Z_i \cap B(x_0, r)$ is at least $\frac{3}{4}2^{r - (i-1)}- (i-1)\varepsilon_n$ if $i \leq r \leq \frac{1}{2}\sys \widehat{M}$, and in particular, that the volume of $B(x_0, r)$ is at least $\frac{3}{2^{n+1}}2^r - (n-1)\varepsilon_n$ if $n \leq r \leq \frac{1}{2}\sys \widehat{M}$.  Taking $\varepsilon_n \rightarrow 0$ gives the desired lower bound for $V_r(\widetilde{M})$.

\bibliography{mu-bib}

\providecommand{\bysame}{\leavevmode\hbox to3em{\hrulefill}\thinspace}
\providecommand{\MR}{\relax\ifhmode\unskip\space\fi MR }
\providecommand{\MRhref}[2]{%
  \href{http://www.ams.org/mathscinet-getitem?mr=#1}{#2}
}
\providecommand{\href}[2]{#2}
\begin{thebibliography}{{Alp}22}

\bibitem[AK16]{alpert16}
Hannah Alpert and Gabriel Katz, \emph{Using simplicial volume to count
  multi-tangent trajectories of traversing vector fields}, Geom. Dedicata
  \textbf{180} (2016), 323--338.

\bibitem[{Alp}22]{alpert22}
Hannah {Alpert}, \emph{{Simplicial volume and 0-strata of separating
  filtrations}}, arXiv e-prints (2022), arXiv:2211.06362, accepted for
  publication in \emph{Journal of Topology and Analysis}.

\bibitem[BCG95]{besson95}
G.~Besson, G.~Courtois, and S.~Gallot, \emph{Entropies et rigidit\'{e}s des
  espaces localement sym\'{e}triques de courbure strictement n\'{e}gative},
  Geom. Funct. Anal. \textbf{5} (1995), no.~5, 731--799.

\bibitem[DKL19]{dey19}
Subhadip Dey, Michael Kapovich, and Beibei Liu, \emph{Ping-pong in {H}adamard
  manifolds}, M\"{u}nster J. Math. \textbf{12} (2019), no.~2, 453--471.

\bibitem[Gro82]{gromov82}
Michael Gromov, \emph{Volume and bounded cohomology}, Inst. Hautes \'{E}tudes
  Sci. Publ. Math. (1982), no.~56, 5--99 (1983).

\bibitem[Gro83]{gromov83}
Mikhael Gromov, \emph{Filling {R}iemannian manifolds}, J. Differential Geom.
  \textbf{18} (1983), no.~1, 1--147.

\bibitem[Gro09]{gromov09}
Mikhail Gromov, \emph{Singularities, expanders and topology of maps. {I}.
  {H}omology versus volume in the spaces of cycles}, Geom. Funct. Anal.
  \textbf{19} (2009), no.~3, 743--841.

\bibitem[Gut11]{guth11}
Larry Guth, \emph{Volumes of balls in large {R}iemannian manifolds}, Ann. of
  Math. (2) \textbf{173} (2011), no.~1, 51--76.

\bibitem[Iva87]{ivanov87}
N.~V. Ivanov, \emph{Foundations of the theory of bounded cohomology}, Journal
  of Soviet Mathematics \textbf{37} (1987), no.~3, 1090--1115.

\bibitem[Kar15]{karam15}
Steve Karam, \emph{Growth of balls in the universal cover of surfaces and
  graphs}, Trans. Amer. Math. Soc. \textbf{367} (2015), no.~8, 5355--5373.

\bibitem[Nab22]{nabutovsky19}
Alexander Nabutovsky, \emph{Linear bounds for constants in {G}romov's systolic
  inequality and related results}, Geom. Topol. \textbf{26} (2022), no.~7,
  3123--3142.

\bibitem[Pap20]{papasoglu20}
Panos Papasoglu, \emph{Uryson width and volume}, Geom. Funct. Anal. \textbf{30}
  (2020), no.~2, 574--587.

\bibitem[Sab22]{sabourau22}
St\'{e}phane Sabourau, \emph{Macroscopic scalar curvature and local
  collapsing}, Ann. Sci. \'{E}c. Norm. Sup\'{e}r. (4) \textbf{55} (2022),
  no.~4, 919--936.

\end{thebibliography}
\bibliographystyle{amsalpha}
\end{document}